\documentclass[10pt,a4paper,reqno]{amsart}
\usepackage[margin=4.258cm]{geometry}
\usepackage{amsmath}
\usepackage{amsthm}
\usepackage{amssymb}
\usepackage{hyperref}
\hypersetup{colorlinks=true,linkcolor=blue,citecolor=blue}
\newcount\minutes \newcount\hours
\hours=\time
\divide\hours 60
\minutes=\hours
\multiply\minutes -60
\advance\minutes \time
\newcommand{\klockan}{\the\hours:{\ifnum\minutes<10 0\fi}\the\minutes}
\newcommand{\tid}{\today\ \klockan}
\newcommand{\prtid}{\smash{\raise 10mm \hbox{\LaTeX ed \tid}}}
\renewcommand{\prtid}{}
\makeatletter
\let\@@evenhead\@evenhead
\let\@@oddhead\@oddhead
\def\@evenhead{\@@evenhead{\footnotesize\rm\llap{\prtid}}}
\def\@oddhead{{\footnotesize\rm\llap{\prtid}}\@@oddhead}
\def\@oddhead{\@@oddhead{\footnotesize\rm\llap{\prtid}}}
\makeatother

\newtheoremstyle{mydefinitionstyle}
  {1.5ex}        % Space above
  {1.5ex}        % Space below
  {\rmfamily}    % Body font
  {}             % Indent amount
  {\bfseries}    % Head font
  {.}            % Punctuation after theorem head
  {.5em}         % Space after theorem head
  {}             % Theorem head spec(?)
\newtheoremstyle{mytheoremstyle}
  {1.5ex}        % Space above
  {1.5ex}        % Space below
  {\itshape}     % Body font
  {}             % Indent amount
  {\bfseries}    % Theorem head font
  {.}            % Punctuation after theorem head
  {.5em}         % Space after theorem head
  {}             % Theorem head spec (can be left empty, meaning ‘normal’)

\theoremstyle{mytheoremstyle}
\newtheorem{theorem}{Theorem}[section]
\newtheorem{lemma}[theorem]{Lemma}
\newtheorem{proposition}[theorem]{Proposition}
\newtheorem{corollary}[theorem]{Corollary}

\theoremstyle{mydefinitionstyle}
\newtheorem{definition}[theorem]{Definition}

\newtheorem{remark}[theorem]{Remark}

\numberwithin{equation}{section}

\makeatletter
\let\Enumerate=\enumerate
\renewcommand{\enumerate}{\Enumerate%
\renewcommand{\theenumi}{\textup{(\alph{enumi})}}%
\renewcommand{\labelenumi}{\theenumi}%
}
\makeatother 
\makeatletter
\def\mdots@{\mathinner.\nonscript\!.
 \ifx\next,.\else\ifx\next;.\else\ifx\next..\else
 \nonscript\!\mathinner.\fi\fi\fi}
\let\ldots\mdots@
\let\cdots\mdots@
\let\dotso\mdots@
\let\dotsb\mdots@
\let\dotsm\mdots@
\let\dotsc\mdots@
\def\vdots{\vbox{\baselineskip2.8\p@ \lineskiplimit\z@
    \kern6\p@\hbox{.}\hbox{.}\hbox{.}\kern3\p@}}
\def\ddots{\mathinner{\mkern1mu\raise8.6\p@\vbox{\kern7\p@\hbox{.}}
    \raise5.8\p@\hbox{.}\raise3\p@\hbox{.}\mkern1mu}}
\def\cdots{\mathinner{\mkern1mu{\cdot}{\cdot}{\cdot}\mkern1mu}}
\makeatother
{\catcode`p =12 \catcode`t =12 \gdef\eeaa#1pt{#1}}
\def\accentadjtext#1{\setbox0\hbox{$#1$}\kern
                \expandafter\eeaa\the\fontdimen1\textfont1 \ht0 }
\def\accentadjscript#1{\setbox0\hbox{$#1$}\kern
                \expandafter\eeaa\the\fontdimen1\scriptfont1 \ht0 }
\def\accentadjscriptscript#1{\setbox0\hbox{$#1$}\kern
                \expandafter\eeaa\the\fontdimen1\scriptscriptfont1 \ht0 }
\def\accentadjtextback#1{\setbox0\hbox{$#1$}\kern
                -\expandafter\eeaa\the\fontdimen1\textfont1 \ht0 }
\def\accentadjscriptback#1{\setbox0\hbox{$#1$}\kern
                -\expandafter\eeaa\the\fontdimen1\scriptfont1 \ht0 }
\def\accentadjscriptscriptback#1{\setbox0\hbox{$#1$}\kern
                -\expandafter\eeaa\the\fontdimen1\scriptscriptfont1 \ht0 }
\def\itoverline#1{{\mathsurround0pt\mathchoice
        {\rlap{$\accentadjtext{\displaystyle #1}
                \accentadjtext{\vrule height1.593pt}
                \overline{\phantom{\displaystyle #1}
                \accentadjtextback{\displaystyle #1}}$}{#1}}
        {\rlap{$\accentadjtext{\textstyle #1}
                \accentadjtext{\vrule height1.593pt}
                \overline{\phantom{\textstyle #1}
                \accentadjtextback{\textstyle #1}}$}{#1}}
        {\rlap{$\accentadjscript{\scriptstyle #1}
                \accentadjscript{\vrule height1.593pt}
                \overline{\phantom{\scriptstyle #1}
                \accentadjscriptback{\scriptstyle #1}}$}{#1}}
        {\rlap{$\accentadjscriptscript{\scriptscriptstyle #1}
                \accentadjscriptscript{\vrule height1.593pt}
                \overline{\phantom{\scriptscriptstyle #1}
                \accentadjscriptscriptback{\scriptscriptstyle #1}}$}{#1}}}}

\def\vint{\mathop{\mathchoice%
          {\setbox0\hbox{$\displaystyle\intop$}\kern 0.22\wd0
           \vcenter{\hrule width 0.6\wd0}\kern -0.82\wd0}
          {\setbox0\hbox{$\textstyle\intop$}\kern 0.2\wd0
           \vcenter{\hrule width 0.6\wd0}\kern -0.8\wd0}
          {\setbox0\hbox{$\scriptstyle\intop$}\kern 0.2\wd0
           \vcenter{\hrule width 0.6\wd0}\kern -0.8\wd0}
          {\setbox0\hbox{$\scriptscriptstyle\intop$}\kern 0.2\wd0
           \vcenter{\hrule width 0.6\wd0}\kern -0.8\wd0}}
          \mathopen{}\int}

\newcommand{\limminus}{{\mathchoice{\raise.17ex\hbox{$\scriptstyle -$}}
                {\raise.17ex\hbox{$\scriptstyle -$}}
                {\raise.1ex\hbox{$\scriptscriptstyle -$}}
                {\scriptscriptstyle -}}}
\newcommand{\limplus}{{\mathchoice{\raise.17ex\hbox{$\scriptstyle +$}}
                {\raise.17ex\hbox{$\scriptstyle +$}}
                {\raise.1ex\hbox{$\scriptscriptstyle +$}}
                {\scriptscriptstyle +}}}

\newcommand{\calK}{\mathcal{K}}
\newcommand{\calM}{\mathcal{M}}
\newcommand{\Q}{\mathbb{Q}}
\newcommand{\R}{\mathbb{R}}

\newcommand{\eR}{{\overline{\R\kern-0.08em}\kern 0.08em}} 
\def\cprime{{\mathsurround0pt$'$}}		

\DeclareMathOperator{\Div}{div}
\DeclareMathOperator{\diam}{diam}
\DeclareMathOperator{\Lip}{Lip}
\newcommand{\Lipc}{{\Lip_c}}
\DeclareMathOperator{\Mod}{Mod}
\newcommand{\Modp}{{\Mod_p}}
\DeclareMathOperator*{\essliminf}{ess\,lim\,inf}

\DeclareMathOperator*{\essinf}{ess\,inf}

\newcommand{\setm}{\setminus}

\renewcommand{\emptyset}{\varnothing}
\newcommand{\Cp}{{C_p}}
\newcommand{\pp}{{$p\mspace{1mu}$}}   
\newcommand{\bdy}{\partial}
\newcommand{\loc}{_{\rm loc}}
\newcommand{\dmu}{d\mu}
\newcommand{\ds}{ds}
\newcommand{\dx}{dx}
\newcommand{\eps}{\varepsilon}
\renewcommand{\phi}{\varphi}
\newcommand{\Lp}{L^{p}}
\newcommand{\Lploc}{L^{p}\loc}
\newcommand{\Lq}{L^{q}}
\newcommand{\Dp}{D^p}
\newcommand{\Dploc}{D^{p}\loc}
\newcommand{\Np}{N^{1,p}}
\newcommand{\Nploc}{N^{1,p}\loc}
\newcommand{\tNp}{\widetilde{N}^{1,p}}
\newcommand{\K}{\calK}
\newcommand{\Kb}{\calK^{\:\!_\mathrm{B}}}
\newcommand{\Hp}{H}

\begin{document}

\author{Daniel Hansevi}  
\address{Department of Mathematics, 
		Link\"oping University, 
		SE--581 83 Link{\"o}ping, 
		Sweden}
\email{daniel.hansevi@liu.se}

\title[The obstacle problem for unbounded sets in metric spaces]{
		The obstacle and Dirichlet problems \\ 
		associated with \pp-harmonic functions \\ 
		in unbounded sets in $\R^n$ and metric spaces}

\begin{abstract}
We study the obstacle problem for unbounded sets 
in a proper metric measure space supporting a $(p,p)$-Poincar\'e inequality. 
We prove that there exists a unique solution. 
We also prove that if the measure is doubling 
and the obstacle is continuous, 
then the solution is continuous, 
and moreover \pp-harmonic in the set where it does not touch the obstacle. 
This includes, as a special case, 
the solution of the Dirichlet problem for \pp-harmonic functions 
with Sobolev type boundary data.
\end{abstract}



\date{\today}

\maketitle

%===============================================================================
\section{Introduction} %========================================================
%===============================================================================
The classical Dirichlet problem is the problem of finding a harmonic function, 
that is, a solution of the Laplace equation 
that takes prescribed boundary values. 
According to Dirichlet's principle, 
this is equivalent to minimizing the Dirichlet energy integral, 
\[
	\int_\Omega|\nabla u|^2\,\dx, 
\]
among all functions $u$, in the domain $\Omega$, 
that have the required boundary values 
and continuous partial derivatives up to the second order. 

A more general (nonlinear) Dirichlet problem 
considers the \pp-Laplace equation, 
\[
	\Delta_p u 
	:= \Div(|\nabla u|^{p-2}\,\nabla u) 
	= 0,\ \
	1<p<\infty 
\]
(which reduces to the Laplace equation when $p=2$). 
Solving this problem is equivalent to 
the variational problem of minimizing the \pp-energy integral, 
\[
	\int_\Omega|\nabla u|^p\,\dx, 
\]
among all admissible functions $u$, 
and a minimizer/solution that is continuous is said to be \emph{\pp-harmonic}. 

The nonlinear potential theory of \pp-harmonic functions 
has been studied since the 1960s. 
Initially for $\R^n$, and later generalized to weighted $\R^n$, 
Riemannian manifolds, and other settings.  
The interested reader may consult the monograph 
Heinonen--Kilpel\"ainen--Martio~\cite{HeKiMa06} 
for a thorough treatment in weighted $\R^n$. 

It is not clear how to employ partial differential equations 
in a general metric measure space. 
However, 
by using the notion of minimal \pp-weak upper gradients, 
as substitutes for the modulus of the usual gradients, 
the variational approach becomes available. 
This has led to the more recent development of nonlinear potential theory 
on complete metric spaces equipped with a 
doubling measure supporting a \pp-Poincar\'e inequality. 

In this paper, 
instead of just studying the Dirichlet problem for \pp-harmonic functions, 
we study the associated obstacle problem 
with a given obstacle and given boundary values. 
We minimize the \pp-energy integral 
among admissible functions lying above the obstacle $\psi$. 
This problem reduces to the Dirichlet problem when $\psi\equiv-\infty$. 
The obstacle problem has been studied 
for \emph{bounded} sets in (weighted) $\R^n$ 
(see, e.g., Heinonen--Kilpel\"ainen--Martio~\cite{HeKiMa06} 
and the references therein) and later also in metric spaces 
(see, e.g., 
Bj\"orn--Bj\"orn~\cite{BjBj06}, \cite{Boken}, \cite{BjBj12a}, 
Bj\"orn--Bj\"orn--M\"ak\"al\"ainen--Parviainen~\cite{BjBjMaPa09}, 
Bj\"orn--Bj\"orn--Shanmugalingam~\cite{BjBjSh03b}, 
Eleuteri--Farnana--Kansanen--Korte~\cite{ElFaKaKo10}, 
Farnana~\cite{Farnana09}, \cite{Farnana10a}, 
\cite{Farnana10b}, \cite{Farnana11}, 
Kinnunen--Martio~\cite{KiMa02}, 
Kinnunen--Shanmugalingam~\cite{KiSh06}, and 
Shanmugalingam~\cite{Shanmugalingam01}).

Suppose that $\Omega$ is a nonempty (possibly unbounded) open subset of 
a proper metric measure space that supports a $(p,p)$-Poincar\'e inequality. 
Furthermore, suppose that the capacity of the complement of $\Omega$ is nonzero 
(this is needed for the boundary data to make sense). 
Let $\psi$ be an extended real-valued function and 
let $f$ be a function in $\Dp(\Omega)$ 
(see Section~\ref{1-section-prel} for definitions). 
In this setting, we prove Theorem~\ref{1-thm-obst-existence-uniqueness}, 
which asserts that there exists a unique 
\textup{(}up to sets of capacity zero\textup{)}
solution of the $\K_{\psi,f}(\Omega)$-obstacle problem 
whenever the space of admissible functions is nonempty. 

Moreover, by adding the assumption of the measure being doubling, 
we obtain Theorem~\ref{1-thm-obst-solve-cont}, 
which, as a special case, implies that there is a unique solution 
of the Dirichlet problem for \pp-harmonic functions 
with boundary values in $\Dp(\Omega)$ taken in Sobolev sense 
(i.e., that the $\K_{\psi,f}(\Omega)$-obstacle problem 
has a unique \emph{continuous} solution 
whenever $\psi\equiv-\infty$). 

To the best of the author's knowledge, 
these results are new also for $\R^n$.

\medskip

%===============================================================================
\section{Notation and preliminaries} 
\label{1-section-prel} 
%===============================================================================
We assume throughout the paper that $(X,\calM,\mu,d)$ 
is a metric measure space (which we will refer to as $X$) 
equipped with a metric $d$ and a measure $\mu$ such that 
\[
	0 < \mu(B) < \infty
\] 
for all balls $B$ in $X$ 
(we make the convention that balls are nonempty and open). 
The $\sigma$-algebra $\calM$ on which $\mu$ is defined 
is the completion of the Borel $\sigma$-algebra. 

We start with the assumption that $1\leq p<\infty$. 
However, in the next section (and for the rest of the paper), 
we will assume that $1<p<\infty$.

The measure $\mu$ is said to be \emph{doubling} if there exists 
a constant $C_\mu\geq 1$ such that 
\[
	0 < \mu(2B) \leq C_\mu\mu(B) < \infty 
\]
for all balls $B$ in $X$. 
We use the notation that if $B$ is a ball with radius $r$, 
then the ball with radius $\lambda r$ that is concentric with $B$ 
is denoted by $\lambda B$.

The characteristic function $\chi_E$ of a set $E$ is defined by 
$\chi_E(x)=1$ if $x\in E$ and $\chi_E(x)=0$ if $x\notin E$. 
The set $E$ is compactly contained in $A$ 
if $\itoverline{E}$ (the closure of $E$) is a compact subset of $A$. 
We denote this by $E\Subset A$. 
The extended real number system is denoted by \smash{$\eR:=[-\infty,\infty]$}. 
Recall that $f_\limplus=\max\{f,0\}$ and $f_\limminus=\max\{-f,0\}$, 
and hence that $f=f_\limplus-f_\limminus$ and $|f|=f_\limplus+f_\limminus$. 

By a \emph{curve} in $X$,  
we mean a rectifiable nonconstant continuous mapping 
$\gamma$ from a compact interval into $X$. 
Since our curves have finite length, 
they may be parametrized by arc length, 
and we will always assume that this has been done. 
We will abuse notation and denote both the mapping and the image by $\gamma$.

Unless otherwise stated, 
the letter $C$ will be used to denote various positive constants 
whose exact values are unimportant and may vary with each usage. 

\medskip

We follow 
Heinonen--Koskela~\cite{HeKo96}, \cite{HeKo98} 
in introducing upper gradients. 
(Heinonen and Koskela, however, called them very weak gradients.) 
%===============================================================================
\begin{definition} \label{1-def-ug}
A Borel function $g\colon X\to[0,\infty]$ is said to be an 
\emph{upper gradient} 
of a function $f\colon X\to\eR$ whenever 
\begin{equation}\label{1-upper-gradient-ineq}
	|f(x) - f(y)| 
	\leq \int_\gamma g\,\ds
\end{equation}
holds for all pairs of points $x,y\in X$ 
and every curve $\gamma$ in $X$ joining $x$ and $y$. 
We make the convention that the left-hand side is infinite 
when at least one of the terms is.
\end{definition}
%===============================================================================
Recall that a Borel function 
$g\colon X\to Y$ is a function such that the inverse image  
$g^{-1}(G)=\{x\in X:g(x)\in G\}$ is a Borel set 
for every open subset $G$ of $Y$. 

Observe that upper gradients are not unique 
(if we add a nonnegative Borel function to an upper gradient of $f$, 
then we obtain a new upper gradient of $f$) 
and that $g\equiv\infty$ is an upper gradient of all functions. 
Note also that if $g$ and $\tilde g$ are upper gradients of 
$u$ and $\tilde u$, respectively, 
then $g-\tilde g$ is not in general an upper gradient of $u-\tilde u$. 
However, upper gradients are subadditive, that is, 
if $g$ and $\tilde g$ are upper gradients of $u$ and $\tilde u$, 
respectively, and $\alpha\in\R$, then 
$|\alpha|g$ and $g+\tilde g$ are upper gradients of 
$\alpha u$ and $u+\tilde u$, respectively. 

A drawback of upper gradients is that 
they are not preserved by $\Lp$-convergence. 
Fortunately, 
it is possible to overcome this problem by relaxing the conditions. 
Therefore, we define the \pp-modulus of a curve family,    
and then follow Koskela--MacManus~\cite{KoMac98} 
in introducing \pp-weak upper gradients.
%===============================================================================
\begin{definition} \label{1-def-pmod}
Let $\Gamma$ be a family of curves in $X$. 
The \emph{\pp-modulus} of $\Gamma$ is  
\[
	\Modp(\Gamma) 
	:= \inf_\rho\int_X\rho^p\,\dmu, 
\]
where the infimum is taken over all nonnegative Borel functions $\rho$ 
such that 
\[
	\int_\gamma\rho\,\ds\geq 1\ \
	\text{for all curves }\gamma\in\Gamma.
\] 

Whenever a property holds for all curves 
except for a curve family of zero \pp-modulus, 
it is said to hold for \emph{\pp-almost every} (\emph{\pp-a.e.}) curve.
\end{definition}
%===============================================================================
The \pp-modulus (as the module of order $p$ of a system of measures) 
was defined and studied by Fuglede~\cite{Fuglede57}. 
Heinonen--Koskela~\cite{HeKo98} defined the \pp-modulus of a curve family 
in a metric measure space and observed that the corresponding results 
by Fuglede carried over directly. 

The \pp-modulus has the following properties (as observed in \cite{HeKo98}): 
$\Modp(\emptyset)=0$, 
$\Modp(\Gamma_1)\leq\Modp(\Gamma_2)$ whenever $\Gamma_1\subset\Gamma_2$, 
and $\Modp\bigl(\bigcup_{j=1}^\infty\Gamma_j\bigr)
		\leq\sum_{j=1}^\infty\Modp(\Gamma_j)$. 
If $\Gamma_0$ and $\Gamma$ are two curve families such that 
every curve $\gamma\in\Gamma$ has a subcurve $\gamma_0\in\Gamma_0$, 
then $\Modp(\Gamma)\leq\Modp(\Gamma_0)$. 
For proofs of these properties and all other results in this section, 
we refer to Bj\"orn--Bj\"orn~\cite{Boken}. 
(Some of the references that we mention below 
may not provide a proof in the generality considered here, 
but such proofs are given in \cite{Boken}.) 
%===============================================================================
\begin{definition} \label{1-def-pwug}
A measurable function $g\colon X\to[0,\infty]$ 
is said to be a \emph{\pp-weak upper gradient} 
of a function $f\colon X\to\eR$ 
if \eqref{1-upper-gradient-ineq} 
holds for all pairs of points $x,y\in X$ 
and \pp-a.e.\ curve $\gamma$ in $X$ 
joining $x$ and $y$. 
\end{definition}
%===============================================================================
Note that a \pp-weak upper gradient, 
as opposed to an upper gradient, 
is \emph{not} required to be a Borel function. 
It is convenient to demand upper gradients to be Borel functions, 
since then the concept of upper gradients becomes independent of the measure, 
and all considered curve integrals will be defined. 
The situation is a bit different for \pp-weak upper gradients, 
as the curve integrals need only be defined for \pp-a.e.\ curve, 
and therefore, 
it is in fact enough to require that \pp-weak upper gradients 
are measurable functions. 
There is no disadvantage in assuming only measurability, 
since the concept of \pp-weak upper gradients 
would depend on the measure anyway
(as the \pp-modulus depends on the measure). 
The advantage is that some results become more appealing 
(see, e.g., Bj\"orn--Bj\"orn~\cite{Boken}).

Since the \pp-modulus is subadditive, 
it follows that \pp-weak upper gradients 
share the subadditivity property with upper gradients. 
%===============================================================================
\begin{definition} \label{1-def-Dp}
The \emph{Dirichlet space} on $X$, denoted by $\Dp(X)$, 
is the space of all extended real-valued functions on $X$ 
that are everywhere defined, measurable, 
and have upper gradients in $\Lp(X)$.
\end{definition}
%===============================================================================
If $E$ is a measurable set, 
then we can consider $E$ to be a metric space in its own right 
(with the restriction of $d$ and $\mu$ to $E$). 
Thus the Dirichlet space $\Dp(E)$ 
is also given by Definition~\ref{1-def-Dp}. 
Note, however, that the collection of upper gradients 
with respect to $E$ can differ from those 
with respect to $X$ (unless $E$ is open).

The local Dirichlet space is defined analogously 
to the local space $\Lploc(X)$. 
Thus we say that a function $f$ on $X$ belongs to $\Dploc(X)$ 
if for every $x\in X$ 
there is a ball $B$ such that 
$x\in B$ and $f\in\Dp(B)$.

Lemma~2.4 in Koskela--MacManus~\cite{KoMac98} asserts that 
if $g$ is a \pp-weak upper gradient of a function $f$, 
then for all $q$ such that $1\leq q\leq p$, there is a decreasing sequence 
$\{g_j\}_{j=1}^\infty$ of upper gradients of $f$ such that  
$\|g_j-g\|_{\Lq(X)}\to 0$ as $j\to\infty$. 
This implies that a measurable function belongs to $\Dp(X)$ 
whenever it (merely) has a \pp-weak upper gradient in $\Lp(X)$. 

If $u$ belongs to $\Dp(X)$, then $u$ has a 
\emph{minimal \pp-weak upper gradient} $g_u\in\Lp(X)$. 
It is minimal in the sense that $g_u\leq g$ a.e.\ 
for all \pp-weak upper gradients $g$ of $u$. 
This was proved for $p>1$ by Shanmugalingam~\cite{Shanmugalingam01}  
and $p\geq 1$ by Haj\l{}asz~\cite{Hajlasz03}. 
Minimal \pp-weak upper gradients $g_u$ are true substitutes 
for $|\nabla u|$ in metric spaces.

One of the important properties of minimal \pp-weak gradients is 
that they are local in the sense that if two functions $u,v\in\Dp(X)$ 
coincide on a set $E$, then $g_u=g_v$ a.e.\ on $E$.
Moreover, if $U=\{x\in X:u(x)>v(x)\}$, 
then $g_u\chi_U+g_v\chi_{X\setm U}$ 
is a minimal \pp-weak upper gradient of $\max\{u,v\}$, 
and $g_v\chi_U+g_u\chi_{X\setm U}$ 
is a minimal \pp-weak upper gradients of $\min\{u,v\}$.
These results are from Bj\"orn--Bj\"orn~\cite{BjBj04}. 

It is well-known that the restriction of a minimal \pp-weak upper gradient 
to an open subset remains minimal with respect to that subset.
As a consequence, 
the results above about minimal \pp-weak upper gradients 
extend to functions in $\Dploc(X)$ 
having minimal \pp-weak upper gradients in $\Lploc(X)$. 

\medskip

With the help of \pp-weak upper gradients, 
it is possible to define a type of Sobolev space on the metric space $X$. 
This was done by Shanmugalingam~\cite{Shanmugalingam00}. 
We will, however, use a slightly different (semi)norm. 
The reason for this is that when we define the capacity in 
Definition~\ref{1-def-cap}, 
it will be subadditive.
%===============================================================================
\begin{definition} \label{1-def-Np}
The \emph{Newtonian space} on $X$ is 
\[
	\Np(X) 
	:= \{u\in\Dploc(X):\|u\|_{\Np(X)}<\infty\}, 
\]
where $\|\cdot\|_{\Np(X)}$ is the seminorm defined by 
\[
	\|u\|_{\Np(X)} 
	= \biggl(\int_X|u|^p\,\dmu + \int_X g_u^p\,\dmu\biggr)^{1/p}.
\]
\end{definition}
%===============================================================================
We emphasize the fact that our Newtonian functions are defined everywhere, 
and not just up to equivalence classes of functions 
that agree almost everywhere.
This is essential for the notion of upper gradients to make sense.

The associated normed space defined by $\tNp(X)=\Np(X)/\sim$, 
where $u\sim v$ if and only if $\|u-v\|_{\Np(X)}=0$, is a Banach space 
(see Shanmugalingam~\cite{Shanmugalingam00}). 
Note that some authors denote the space of the 
everywhere defined functions by $\tNp(X)$, 
and then define the Newtonian space, 
which they denote by $\Np(X)$, 
to be the corresponding space of equivalence classes.

The local space $\Nploc(X)$ and the space $\Np(E)$ 
when $E$ is a measurable set are defined analogously to the Dirichlet spaces.

Recall that a metric space is said to be \emph{proper} 
if all bounded closed subsets are compact. 
In particular, this is true if it is complete and the measure is doubling. 
If $X$ is proper and $\Omega$ is an open subset of $X$, 
then $f\in\Lploc(\Omega)$ 
if and only if 
$f\in\Lp(\Omega')$ 
for all open $\Omega'\Subset \Omega$. 
This is the case also for \smash{$\Dploc$} and \smash{$\Nploc$}.

\medskip

Various definitions of capacities for sets can be found in the literature 
(see, e.g., Kinnunen--Martio~\cite{KiMa96} and 
Shanmugalingam~\cite{Shanmugalingam00}). 
We will use the following definition.
%===============================================================================
\begin{definition} \label{1-def-cap} 
The (\emph{Sobolev}) \emph{capacity} of a subset $E$ of $X$ 
is 
\[
	\Cp(E) 
	:= \inf_u\|u\|_{\Np(X)}^p,
\]
where the infimum is taken over all 
$u\in\Np(X)$ such that $u\geq 1$ on $E$.
\end{definition}
%===============================================================================
Whenever a property holds for all points 
except for points in a set of capacity zero, 
it is said to hold \emph{quasieverywhere} (\emph{q.e.}). 
Note that we follow the custom of refraining 
from making the dependence on $p$ explicit here.

Trivially, we have $\Cp(\emptyset)=0$, 
and $\Cp(E_1)\leq\Cp(E_2)$ whenever $E_1\subset E_2$. 
Furthermore, the proof in Kinnunen--Martio~\cite{KiMa96} for 
capacities for Haj\l{}asz--Sobolev spaces on metric spaces can easily 
be modified to show that $\Cp$ is countably subadditive, 
that is, 
$\Cp\bigl(\bigcup_{j=1}^\infty E_j\bigr)\leq\sum_{j=1}^\infty\Cp(E_j)$. 
Thus $\Cp$ is an outer measure. 
Note that $C_p$ is finer than $\mu$ 
in the sense that the capacity of a set may be positive 
even when the measure of the same set equals zero.

Shanmugalingam~\cite{Shanmugalingam00} 
showed that if two Newtonian functions are equal almost everywhere, 
then they are in fact equal quasieverywhere.
This result extends to functions in $\Dploc(X)$. 

When $E$ is a subset of $X$, 
we let $\Gamma_E$ denote the family of all curves in $X$ 
that intersect $E$. 
Lemma~3.6 in Shanmugalingam~\cite{Shanmugalingam00} 
asserts that $\Modp(\Gamma_E)=0$ whenever $\Cp(E)=0$. 
This implies that two functions  
have the same set of \pp-weak upper gradients 
whenever they are equal quasieverywhere.

\medskip

In order to be able to compare boundary values of Dirichlet functions 
(and Newtonian functions), 
we introduce the following spaces.
%===============================================================================
\begin{definition} \label{1-def-Dp0}
The \emph{Dirichlet space with zero boundary values in $A\setm E$}, 
for subsets $E$ and $A$ of $X$, where $A$ is measurable, 
is 
\[
	\Dp_0(E;A)
	:= \{f|_{E\cap A}:f\in\Dp(A) 
		\text{ and } f=0 \text{ in } A\setm E\}.
\]
The \emph{Newtonian space with zero boundary values in $A\setm E$}, 
denoted by $\Np_0(E;A)$,  
is defined analogously. 

We let $\Dp_0(E)$ and $\Np_0(E)$ denote 
$\Dp_0(E;X)$ and $\Np_0(E;X)$, respectively. 
\end{definition}
%===============================================================================
The assumption ``$f=0$ in $A\setm E$'' can in fact be replaced by 
``$f=0$ q.e.\ in $A\setm E$'' without changing the obtained spaces.

It is easy to verify that the function spaces that we have introduced 
are vector spaces and lattices. 
This means that if $u,v\in\Dp(X)$ and $a,b\in\R$, 
then we have $au+bv,\max\{u,v\},\min\{u,v\}\in\Dp(X)$, 
and furthermore, as a direct consequence, 
we also have $u_\limplus,u_\limminus,|u|\in\Dp(X)$. 

The following lemma is useful for asserting that certain functions 
belong to a Dirichlet space with zero boundary values. 
%===============================================================================
\begin{lemma} \label{1-lem-Dp0-police}
Suppose that $E$ is a measurable subset of $X$ and that $u\in\Dp(E)$. 
If there exist two functions $u_1$ and $u_2$ in $\Dp_0(E)$ 
such that $u_1\leq u\leq u_2$ q.e.\ in $E$\textup{,} 
then $u\in\Dp_0(E)$.
\end{lemma}
%-------------------------------------------------------------------------------
This was proved for Newtonian functions in open sets in 
Bj\"orn--Bj\"orn~\cite{BjBj06}, 
and with trivial modifications, 
it provides a proof for our version of the lemma. 
For the reader's convenience, we give the proof here.
%-------------------------------------------------------------------------------
\begin{proof}
Let $v_1$ and $v_2$ be functions in $\Dp(X)$ such that 
$v_1|_E=u_1$, $v_2|_E=u_2$, 
and $v_1=v_2=0$ outside $E$, 
and let $g_1\in\Lp(X)$ and $g_2\in\Lp(X)$ be upper gradients 
of $v_1$ and $v_2$, respectively.
Let $g\in\Lp(E)$ be an upper gradient of $u$ and define  
\[
	v = \begin{cases}
		u & \text{in } E, \\
		0 & \text{in } X\setm E
	\end{cases}
	\quad\text{and}\quad
	\tilde g = \begin{cases}
		g_1+g_2+g & \text{in } E, \\
		g_1+g_2 & \text{in } X\setm E.
	\end{cases}
\]
To complete the proof, 
it suffices to show that 
$\tilde g\in\Lp(X)$ is a \pp-weak upper gradient of $v$. 

Let $E'$ be a subset of $E$ 
with $\Cp(E')=0$ and such that $u_1\leq u\leq u_2$ in $E\setm E'$. 
Let $\gamma$ be an arbitrary curve in $X\setm E'$ with endpoints $x$ and $y$. 
Then $\Modp(\Gamma_{E'})=0$, 
so the following argument asserts that 
$\tilde g$ is a \pp-weak upper gradient of $v$. 

If $\gamma\subset E\setm E'$, then 
\[
	|v(x)-v(y)|
	= |u(x)-u(y)|
	\leq \int_\gamma g\,\ds
	\leq \int_\gamma\tilde g\,\ds.
\]
On the other hand, 
if $x,y\in X\setm E$, 
then 
\[
	|v(x)-v(y)|
	= 0
	\leq \int_\gamma\tilde g\,ds.
\]
Hence, by splitting $\gamma$ into two parts, 
and possibly reversing the direction, 
we may assume that $x\in E\setm E'$ 
and $y\in X\setm E$. 
Then it follows that 
\begin{align*}
	|v(x)-v(y)|
	&= |u(x)|
	\leq |v_1(x)|+|v_2(x)| 
	= |v_1(x)-v_1(y)|
		+|v_2(x)-v_2(y)| \\
	&\leq \int_\gamma g_1\,ds + \int_\gamma g_2\,\ds 
	\leq \int_\gamma\tilde g\,\ds.
	\qedhere
\end{align*}
\end{proof}
%===============================================================================
\begin{proposition} \label{1-prop-Dp0-Om-clOm}
Let\/ $\Omega$ be an open subset of $X$. 
Then $\Dp_0(\Omega)=\Dp_0(\Omega;\overline\Omega)$.
\end{proposition}
%-------------------------------------------------------------------------------
The proof is very similar to the proof of Lemma~\ref{1-lem-Dp0-police} 
(see, e.g., Proposition~2.39 in Bj\"orn--Bj\"orn~\cite{Boken} 
for a corresponding proof 
for Newtonian functions).

\medskip

The next two results from Bj\"orn--Bj\"orn--Parviainen~\cite{BjBjPa10} 
(Lemma~3.2 and Corollary~3.3), 
following from Mazur's lemma (see, e.g., Theorem~3.12 in Rudin~\cite{Rudin91}), 
will play a major role in the existence proof for the obstacle problem.
%===============================================================================
\begin{lemma} \label{1-lem-mazur-consequence}
Assume that\/ $1<p<\infty$. 
Assume further that $g_j$ is a \pp-weak upper gradient of $u_j$\textup{,} 
$j=1,2,\dots$\,\textup{,} 
and that\/ $\{u_j\}_{j=1}^\infty$ and\/ $\{g_j\}_{j=1}^\infty$ 
are bounded in $\Lp(X)$. 
Then there exist functions $u$ and $g$\textup{,} 
both in $\Lp(X)$\textup{,} 
convex combinations \smash{$v_j=\sum_{i=j}^{N_j}a_{j,i}u_i$} 
with \pp-weak upper gradients 
\smash{$\tilde g_j=\sum_{i=j}^{N_j}a_{j,i}g_i$}\textup{,} 
$j=1,2,\dots$\,\textup{,} 
and a subsequence\/ $\{u_{j_k}\}_{k=1}^\infty$\textup{,} 
such that 
\begin{enumerate}
\item both $u_{j_k}\to u$ and 
	$g_{j_k}\to g$ weakly in $\Lp(X)$ as $k\to\infty$\textup{;} 
\item both $v_j\to u$ and $\tilde g_j\to g$ in $\Lp(X)$ as 
	$j\to\infty$\textup{;}
\item $v_j\to u$ q.e.\ as $j\to\infty$\textup{;}
\item $g$ is a \pp-weak upper gradient of $u$.
\end{enumerate}
\end{lemma}
%===============================================================================
Recall that $\alpha_1 v_1+\cdots+\alpha_n v_n$ 
is said to be a convex combination of $v_1,\dots,v_n$ whenever  
$\alpha_k\geq 0$ for all $k=1,\dots,n$ and $\alpha_1+\cdots+\alpha_n=1$. 
%===============================================================================
\begin{corollary} \label{1-cor-mazur-consequence}
Assume that\/ $1<p<\infty$. 
Assume also that\/ $\{u_j\}_{j=1}^\infty$ is bounded in $\Np(X)$ 
and that $u_j\to u$ q.e.\ on $X$ as $j\to\infty$.
Then $u\in\Np(X)$ and 
\[
	\int_X g_u^p\,\dmu 
	\leq \liminf_{j\to\infty}\int_X g_{u_j}^p\,\dmu.
\]
\end{corollary}
%===============================================================================

\medskip

In general, the upper gradients of a function 
give no control over the function. 
This is obviously so when there are no curves. 
Requiring a Poincar\'e inequality to hold 
is one possibility of gaining such a control 
by making sure that there are enough curves connecting any two points.
%===============================================================================
\begin{definition} \label{1-def-poincare-inequality}
Let $q\geq 1$. 
We say that $X$ supports a $(q,p)$-\emph{Poincar\'e inequality} 
(or that $X$ is a $(q,p)$-Poincar\'e space) 
if there exist constants 
$C_{\mathrm{PI}}>0$ and $\lambda\geq 1$ (dilation constant) 
such that for all 
balls $B$ in $X$, 
all integrable functions $u$ on $X$, 
and all upper gradients $g$ of $u$, 
it is true that 
\[
	\biggl(\vint_B|u-u_B|^q\,\dmu\biggr)^{1/q} 
	\leq C_{\mathrm{PI}}\diam(B)
			\biggl(\vint_{\lambda B}g^p\,\dmu\biggr)^{1/p},
\]
where 
\[
	u_B 
	:= \vint_B u\,\dmu 
	:= \frac{1}{\mu(B)}\int_B u\,\dmu.
\] 
For short, 
we say \emph{\pp-Poincar\'e inequality} 
instead of $(1,p)$-Poincar\'e inequality, 
and if $X$ supports a \pp-Poincar\'e inequality, 
we say that $X$ is a \pp-Poincar\'e space. 
\end{definition}
%===============================================================================
By using H\"older's inequality, one can show that 
if $X$ supports a $(q,p)$-Poincar\'e inequality, 
then $X$ supports a $(\tilde q,\tilde p)$-Poincar\'e inequality 
for all $\tilde q\leq q$ and $\tilde p\geq p$.
From the next section on, 
we will assume $X$ to support a $(p,p)$-Poincar\'e inequality. 
Then we have the following useful assertion 
that implies that a function can be controlled 
by its minimal \pp-weak upper gradient. 
This was proved for Euclidean spaces by 
Maz\cprime ya (see, e.g., \cite{Mazya85}), and later 
J.~Bj\"orn~\cite{BjornJ02} observed that the proof 
goes through also for metric spaces. 
The following version is from Bj\"orn--Bj\"orn~\cite{Boken} (Theorem~5.53).
%===============================================================================
\begin{theorem}[Maz\cprime ya's inequality.] \label{1-thm-Mazyas-ineq}
Suppose that $X$ supports a $(p,p)$-Poincar\'e inequality. 
Then there exists a constant\/ $C_{\mathrm{MI}}>0$ such that 
if $B$ is a ball in $X$\textup{,} 
$u\in\Nploc(X)$\textup{,} and $S=\{x\in X:u(x)=0\}$\textup{,} 
then 
\[
	\int_{2B}|u|^p\,\dmu 
	\leq \frac{C_{\mathrm{MI}}(\diam{(B)}^p+1)\mu(2B)}{\Cp(B\cap S)}
		\int_{2\lambda B}g_u^p\,\dmu, 
\]
where $\lambda$ is the dilation constant in the 
$(p,p)$-Poincar\'e inequality.
\end{theorem}
%===============================================================================
The following result from Bj\"orn--Bj\"orn~\cite{Boken} 
(Proposition~4.14) 
is also a useful consequence of the $(p,p)$-Poincar\'e inequality. 
%===============================================================================
\begin{proposition} \label{1-prop-Dploc-Nploc}
Suppose that $X$ supports a $(p,p)$-Poincar\'e inequality. 
Let\/ $\Omega$ be an open subset of $X$. 
Then $\Dploc(\Omega)=\Nploc(\Omega)$.
\end{proposition}

\medskip

%===============================================================================
\section{The obstacle problem} 
\label{1-section-obst} 
%===============================================================================
\emph{In this section\textup{,} 
we assume that\/ $1<p<\infty$\textup{,} 
that $X$ is proper and supports a $(p,p)$-Poincar\'e inequality 
with dilation constant $\lambda$\textup{,} 
and that\/ $\Omega$ is a nonempty open subset of $X$ 
such that $\Cp(X\setm\Omega)>0$.}

\medskip

Kinnunen--Martio~\cite{KiMa02} defined an obstacle problem for 
Newtonian functions in open sets in 
a complete \pp-Poincar\'e space with a doubling measure. 
They proved that there exists a unique solution 
whenever the set is bounded and such that the complement has nonzero measure 
and the set of feasible solutions is nonempty 
(Theorem~3.2 in \cite{KiMa02}). 
Shanmugalingam~\cite{Shanmugalingam00} had earlier solved 
the Dirichlet problem 
(i.e., the obstacle problem with obstacle $\psi\equiv -\infty$).

Roughly, Kinnunen and Martio defined their obstacle as follows. 
%===============================================================================
\begin{definition} \label{1-def-bounded-obst}
Suppose that $V$ is a nonempty \emph{bounded} open subset of $X$ with  
$\Cp(X\setm V)>0$. 
Let $\psi\colon V\to\eR$ and let $f\in\Np(V)$. 
Define 
\[
	\Kb_{\psi,f}(V) 
	= \{v\in\Np(V):v-f\in\Np_0(V)
		\text{ and }v\geq\psi\text{ q.e.\ in }V\}.
\]
Then $u$ is said to be a 
\emph{solution of the} 
$\Kb_{\psi,f}(V)$-\emph{obstacle problem} 
if $u\in\Kb_{\psi,f}(V)$ and 
\[
	\int_V g_u^p\,\dmu 
	\leq \int_V g_v^p\,\dmu\ \
	\text{for all }v\in\Kb_{\psi,f}(V).
\]
\end{definition}
%===============================================================================
They required that 
$\mu(X\setm V)>0$ and merely that $v\geq\psi$ a.e.\ instead of q.e. 
This does not matter if the obstacle $\psi$ is in $\Dploc(V)$, 
since then $v\geq\psi$ a.e.\ implies that $v\geq\psi$ q.e. 
This follows from Corollary~3.3 in Shanmugalingam~\cite{Shanmugalingam00}; 
see also Corollary~1.60 in Bj\"orn--Bj\"orn~\cite{Boken}. 
However, the distinction may be important. 
For example, if $K$ is a compact subset of $V$ such that $\Cp(K)>\mu(K)=0$, 
then the solution of the \smash{$\Kb_{\chi_K,0}(V)$}-obstacle problem 
takes the value $1$ on $K$, 
whereas the solution of the corresponding obstacle problem 
defined by Kinnunen--Martio~\cite{KiMa02} is the trivial solution 
(because their candidate solutions do not ``see'' this obstacle). 
Moreover, it is possible to have no solution of 
the \smash{$\Kb_{\psi,f}(V)$}-obstacle problem when 
there is a solution of the corresponding obstacle problem 
defined by \cite{KiMa02} 
(see, e.g., the discussion following 
Definition~3.1 in Farnana~\cite{Farnana09}).

See also Farnana~\cite{Farnana09}, \cite{Farnana10a}, 
\cite{Farnana10b}, \cite{Farnana11} 
for the double obstacle problem, 
and Bj\"orn--Bj\"orn~\cite{BjBj12a} for obstacle problems on nonopen sets.

\medskip

Now we define our obstacle problem (without the boundedness requirement).
%===============================================================================
\begin{definition} \label{1-def-obst}
Suppose that $V$ is a nonempty (possibly unbounded) 
open subset of $X$ such that $\Cp(X\setm V)>0$. 
Let $\psi\colon V\to\eR$ and let $f\in\Dp(V)$. 
Define 
\[
	\K_{\psi,f}(V)
	= \{v\in\Dp(V):v-f\in\Dp_0(V)
		\text{ and }v\geq\psi\text{ q.e.\ in }V\}.
\]
We say that $u$ is a 
\emph{solution of the }$\K_{\psi,f}(V)$-\emph{obstacle problem 
\textup{(}with obstacle $\psi$ and boundary values $f$\textup{)}}
if $u\in\K_{\psi,f}(V)$ and 
\[
	\int_V g_u^p\,\dmu 
	\leq \int_V g_v^p\,\dmu\ \
	\text{for all }v\in\K_{\psi,f}(V).
\]
When $V=\Omega$, 
we denote $\K_{\psi,f}(\Omega)$ by $\K_{\psi,f}$ for short.
\end{definition}
%===============================================================================
Observe that we only define the obstacle problem for $V$ 
with $\Cp(X\setm V)>0$. 
This is because the condition $u-f\in\Dp_0(V)$ 
becomes empty when $\Cp(X\setm V)=0$, 
since then we have $\Dp_0(V)=\Dp(V)$. 

Note also that we solve the obstacle problem 
for boundary data $f\in\Dp(V)$. 
Since such a function is not defined on $\bdy V$,  
we do not really have boundary values, 
and hence the definition should be understood in a weak Sobolev sense. 
%===============================================================================
\begin{remark}
\label{1-rem-extension} 
If $V$ is bounded, 
then Proposition~2.7 in Bj\"orn--Bj\"orn~\cite{BjBj12a} asserts that 
$\Dp_0(V)=\Np_0(V)$, 
and hence we have 
\smash{$\K_{\psi,f}(V)=\Kb_{\psi,f}(V)$}. 
Thus Definition~\ref{1-def-obst} is a generalization of 
Definition~\ref{1-def-bounded-obst}
to Dirichlet functions and to unbounded sets. 
\end{remark}
%===============================================================================
The main result in this paper shows that 
the $\K_{\psi,f}$-obstacle problem has a unique solution 
under the natural condition of 
$\K_{\psi,f}$ being nonempty.
%===============================================================================
\begin{theorem} \label{1-thm-obst-existence-uniqueness} 
Let $\psi\colon\Omega\to\eR$ and let $f\in\Dp(\Omega)$. 
Then there exists a unique 
\textup{(}up to sets of capacity zero\textup{)}
solution of the $\K_{\psi,f}$-obstacle problem 
whenever $\K_{\psi,f}$ is nonempty.
\end{theorem}
%-------------------------------------------------------------------------------
The assumption that $X$ is proper is needed 
only in the end of the existence part of the proof. 

In the uniqueness part of the proof, 
we use the fact that $\Lp(\Omega)$ is strictly convex. 
Clarkson~\cite{Clarkson36} introduced the 
notions of strict convexity and uniform convexity 
(the latter being a stronger condition), 
and proved that all $\Lp$-spaces, $1<p<\infty$, are uniformly convex. 
A Banach space $Y$ (with norm $\|\cdot\|$) is \emph{strictly convex} 
if $x=cy$ for some constant $c>0$ 
whenever $x$ and $y$ are nonzero 
and $\|x+y\|=\|x\|+\|y\|$. 
In particular, 
$x=y$ whenever 
$\|x\|=\|y\|=\big\|\frac{1}{2}(x+y)\big\|=1$. 

The idea used in the uniqueness part of the proof 
comes from Cheeger~\cite{Cheeger99}.
%-------------------------------------------------------------------------------
\begin{proof}
(Existence.) 
We start by choosing a ball $B\subset X$ such that $\Cp(B\setm\Omega)>0$ 
and $B\cap\Omega$ is nonempty. 
Clearly, we have 
$B\subset 2B\subset 3B\subset\cdots\subset X = \bigcup_{t=1}^\infty tB$.

Let 
\[
	I 
	= \inf_v\int_\Omega g_v^p\,\dmu,
\]
with the infimum taken over all $v\in\K_{\psi,f}$. 
Then $0\leq I<\infty$ as $\K_{\psi,f}$ is nonempty. 
Let $\{u_j\}_{j=1}^\infty\subset\K_{\psi,f}$ 
be a minimizing sequence such that 
\[
	I_j 
	:= \int_\Omega g_{u_j}^p\,\dmu 
	\searrow I\ \
	\text{as }j\to\infty. 
\]
Let $w_j\in\Dp(X)$ be such that $w_j=u_j-f$ in $\Omega$ 
and $w_j=0$ outside $\Omega$, $j=1,2,\dots$\,. 
We claim that both $\{w_j\}_{j=1}^\infty$ and $\{g_{w_j}\}_{j=1}^\infty$ 
are bounded in $\Lp(tB)$ for all $t\geq 1$. 
To show that, 
we first observe that $g_{w_j}\leq(g_{u_j}+g_f)\chi_\Omega$ a.e., 
and hence 
\[
	\|g_{w_j}\|_{\Lp(X)} 
	\leq \|g_{u_j}\|_{\Lp(\Omega)}+\|g_f\|_{\Lp(\Omega)} 
	\leq \|g_{u_1}\|_{\Lp(\Omega)}+\|g_f\|_{\Lp(\Omega)} 
	=: C'
	< \infty.
\]
Let $t\geq 1$ be arbitrary 
and let $S=\bigcap_{j=1}^\infty\{x\in X:w_j(x)=0\}$. 
Then
\[
	\Cp(tB\cap S)
	\geq \Cp(tB\setm\Omega)
	\geq \Cp(B\setm\Omega)
	> 0. 
\]
Maz\cprime ya's inequality (Theorem~\ref{1-thm-Mazyas-ineq}) 
asserts the existence of a constant $C_{tB}>0$ such that
\[
	\int_{2tB}|w_j|^p\,\dmu 
	\leq C^p_{tB}\int_{2\lambda tB}g_{w_j}^p\,\dmu. 
\]
This implies that we also have 
\begin{equation}\label{1-thm-obst-existence-uniqueness-ineq-1} 
	\|w_j\|_{\Lp(tB)}
	\leq C_{tB}\|g_{w_j}\|_{\Lp(X)} 
	\leq C_{tB}C' 
	=: C_{tB}' 
	< \infty,
\end{equation}
and the claim follows.

Consider the ball $B$. 
Lemma~\ref{1-lem-mazur-consequence} asserts that we can find 
a function $\phi_1\in\Lp(B)$ and convex combinations 
\begin{equation}\label{1-thm-obst-existence-uniqueness-eq-1} 
	\phi_{1,j} 
	= \sum_{k=j}^{N_{1,j}}a_{1,j,k}w_k\ \
	\text{in }\Dp(X),\ \
	j=1,2,\dots,
\end{equation}
such that $\phi_{1,j}\to\phi_1$ q.e.\ in $B$ as $j\to\infty$. 
Because $\phi_{1,j}=0$ outside $\Omega$, 
we must have $\phi_1=0$ q.e.\ in $B\setm\Omega$, 
and hence we may choose $\phi_1$ 
so that $\phi_1=0$ in $B\setm\Omega$.
Let $v_{1,j}=f+\phi_{1,j}|_\Omega$. 
Then  
\[
	v_{1,j} 
	= f+\sum_{k=j}^{N_{1,j}}a_{1,j,k}w_k|_\Omega
	= \sum_{k=j}^{N_{1,j}}a_{1,j,k}(f+w_k|_\Omega)
	= \sum_{k=j}^{N_{1,j}}a_{1,j,k}u_k
	\geq \psi\ \
	\text{q.e.\ in }\Omega.
\]
We also have 
\[
	g_{v_{1,j}} 
	\leq \sum_{k=j}^{N_{1,j}}a_{1,j,k}g_{u_k}\ \
	\text{a.e.\ in }\Omega
	\quad\text{and}\quad
	g_{\phi_{1,j}} 
	\leq \sum_{k=j}^{N_{1,j}}a_{1,j,k}g_{w_k}\ \
	\text{a.e.}
\]
A sequence of convex combinations 
of functions taken from a bounded sequence must also 
be bounded, 
and therefore 
we can apply Lemma~\ref{1-lem-mazur-consequence} repeatedly here. 
Hence, 
for every $n=2,3,4,\dots$\,, 
we can find a function $\phi_n\in\Lp(nB)$ 
such that $\phi_n=0$ in $nB\setm\Omega$ 
and convex combinations 
\begin{equation}\label{1-thm-obst-existence-uniqueness-eq-2} 
	\phi_{n,j} 
	= \sum_{k=j}^{N_{n,j}}a_{n,j,k}\phi_{n-1,k}\ \
	\text{in }\Dp(X),\ \
	j=1,2,\dots,
\end{equation}
such that $\phi_{n,j}\to\phi_n$ q.e.\ in $nB$ as $j\to\infty$. 
Let $v_{n,j}=f+\phi_{n,j}|_\Omega$. 
Then 
\[
	v_{n,j} 
	= \sum_{k=j}^{N_{n,j}}a_{n,j,k}(f+\phi_{n-1,k}|_\Omega) 
	= \sum_{k=j}^{N_{n,j}}a_{n,j,k}v_{n-1,k} 
	\geq \psi\ \
	\text{q.e.\ in }\Omega, 
\]
and also  
\[
	g_{v_{n,j}} \leq \sum_{k=j}^{N_{n,j}}a_{n,j,k}g_{v_{n-1,k}}\ \
	\text{a.e.\ in }\Omega
	\quad\text{and}\quad
	g_{\phi_{n,j}} 
	\leq \sum_{k=j}^{N_{n,j}}a_{n,j,k}g_{\phi_{n-1,k}}\ \
	\text{a.e.}
\]

Let $u=f+\phi|_\Omega$, 
where $\phi$ is the function on $X$ defined by 
\[
	\phi(x) 
	= \sum_{n=1}^\infty\phi_n(x)\chi_{nB\setm(n-1)B}(x),\ \
	x\in X.
\]
We shall now show that $u$ is indeed a solution of the 
$\K_{\psi,f}$-obstacle problem. 
To do that, we first establish that $u\in\K_{\psi,f}$, 
and then show that $u$ is a minimizer. 
Because $\phi=u-f$ in $\Omega$ and $\phi=0$ outside $\Omega$, 
it suffices to show that $\phi\in\Dp(X)$ 
in order to establish that $u-f\in\Dp_0(\Omega)$ and $u\in\Dp(\Omega)$.

Consider the diagonal sequences 
$\{v_{n,n}\}_{n=1}^\infty$ and $\{\phi_{n,n}\}_{n=1}^\infty$.
Observe that the latter is bounded in $\Lp(tB)$ for $t\geq 1$, 
since $\|\phi_{n,j}\|_{\Lp(tB)}\leq C_{tB}'$ for all $n$ and $j$, 
by \eqref{1-thm-obst-existence-uniqueness-ineq-1}, 
\eqref{1-thm-obst-existence-uniqueness-eq-1}, 
and \eqref{1-thm-obst-existence-uniqueness-eq-2}.

We claim that $\phi_{n,n}\to\phi$ q.e.\ as $n\to\infty$. 
To prove that, we start by fixing an integer $n\geq 1$ and consider $nB$. 
Then   
\begin{align*}
	|\phi_{n+1}-\phi_n| 
	&\leq |\phi_{n+1}-\phi_{n+1,j}|
		+|\phi_{n+1,j}-\phi_n| \\
	&\leq |\phi_{n+1}-\phi_{n+1,j}|
		+\!\!\sum_{k=j}^{N_{n+1,j}}a_{n+1,j,k}
		|\phi_{n,k}-\phi_n|\to 0  
\end{align*}
q.e.\ in $nB$ as $j\to\infty$. 
Thus $\phi_{n+1}=\phi_n$ q.e.\ in $nB$ for $n=1,2,\dots$\,.

By definition, we have $\phi=\phi_1$ in $B$. 
Now assume that $\phi=\phi_n$ q.e.\ in $nB$ for some positive integer $n$. 
By definition also, 
we have $\phi=\phi_{n+1}$ in $(n+1)B\setm nB$, 
and because $\phi_{n+1}=\phi_n$ q.e.\ in $nB$, 
it follows that $\phi=\phi_{n+1}$ q.e.\ in $(n+1)B$. 
Hence, by induction, 
we have $\phi=\phi_n$ q.e.\ in $nB$ for $n=1,2,\dots$\,.

For $n=1,2,\dots$\,, 
let $E_n$ be the subset of $nB$ where 
$\phi_{n,j}\to\phi_n=\phi$ as $j\to\infty$ 
and let $E=\bigcup_{n=1}^\infty(nB\setm E_n)$.
Then we have $\Cp(E)\leq\sum_{n=1}^\infty\Cp(nB\setm E_n)=0$. 
Let $x\in X\setm E$. 
Clearly, $x\in mB$ 
and $\phi(x)=\phi_m(x)$ for some positive integer $m$. 
Given $\eps>0$, choose a $J$ such that $j\geq J$ implies that  
\[
	|\phi_{m,j}(x)-\phi_{m}(x)| 
	< \eps.
\] 
Assume that for some $n\geq m$, 
we have $|\phi_{n,j}(x)-\phi_{m}(x)|<\eps$ for $j\geq J$.
Then  
\[
	|\phi_{n+1,j}(x)-\phi_m(x)|	
	\leq \sum_{k=j}^{N_{n+1,j}}a_{n+1,j,k}
		|\phi_{n,k}(x)-\phi_m(x)|
	< \eps 
\]
for $j\geq J$. 
By induction, it follows that $|\phi_{n,j}(x)-\phi_m(x)|<\eps$ 
for $n\geq m$ and $j\geq J$, and hence, for $n\geq\max\{m,J\}$, 
we have  
\[
	|\phi_{n,n}(x)-\phi(x)| 
	= |\phi_{n,n}(x)-\phi_m(x)|
	< \eps.
\]
We conclude that $\phi_{n,n}\to\phi$ q.e., 
and also that $v_{n,n}\to u$ q.e.\ in $\Omega$, 
as $n\to\infty$. 

By using Jensen's inequality, we can see that 
\[
	\int_\Omega g_{v_{1,j}}^p\,\dmu 
	\leq \int_\Omega \biggl(
		\sum_{k=j}^{N_{1,j}}a_{1,j,k}g_{u_k}\biggr)^p\dmu 
	\leq \sum_{k=j}^{N_{1,j}}a_{1,j,k}
		\int_\Omega g_{u_k}^p\,\dmu 
	\leq \int_\Omega g_{u_j}^p\,\dmu
\]
and 
\begin{align*}
	\int_X g_{\phi_{1,j}}^p\,\dmu 
	&\leq \int_X \biggl(
		\sum_{k=j}^{N_{1,j}}a_{1,j,k}g_{w_k}\biggr)^p\dmu 
	\leq \sum_{k=j}^{N_{1,j}}a_{1,j,k}
		\int_\Omega(g_{u_k}+g_f)^p\,\dmu \\
	&\leq 2^p\sum_{k=j}^{N_{1,j}}a_{1,j,k}
		\int_\Omega(g_{u_k}^p+g_f^p)\,\dmu 
	\leq 2^p\int_\Omega(g_{u_j}^p+g_f^p)\,\dmu.
\end{align*}
Assume that for some positive integer $n$, it is true that 
\[
	\int_\Omega g_{v_{n,j}}^p\,\dmu
	\leq \int_\Omega g_{u_j}^p\,\dmu
	\quad\text{and}\quad 
	\int_X g_{\phi_{n,j}}^p\,\dmu
	\leq 2^p\int_\Omega(g_f^p+g_{u_j}^p)\,\dmu.
\]
Then 
\begin{align*}
	\int_\Omega g_{v_{n+1,j}}^p\,\dmu
	&\leq \int_\Omega \biggl(
		\sum_{k=j}^{N_{n+1,j}}
		a_{n+1,j,k}g_{v_{n,k}}\biggr)^p \dmu 
	\leq \sum_{k=j}^{N_{n+1,j}}a_{n+1,j,k}
		\int_\Omega g_{v_{n,k}}^p\,\dmu \\
	&\leq \sum_{k=j}^{N_{n+1,j}}a_{n+1,j,k}
		\int_\Omega g_{u_k}^p\,\dmu 
	\leq \int_\Omega g_{u_j}^p\,\dmu
\end{align*}
and 
\begin{align*}
	\int_X g_{\phi_{n+1,j}}^p\,\dmu
	&\leq \int_X \biggl(
		\sum_{k=j}^{N_{n+1,j}}
		a_{n+1,j,k}g_{\phi_{n,k}}\biggr)^p \dmu 
	\leq \sum_{k=j}^{N_{n+1,j}}
		a_{n+1,j,k}\int_X g_{\phi_{n,k}}^p \dmu \\
	&\leq 2^p\sum_{k=j}^{N_{n+1,j}}
		a_{n+1,j,k}\int_\Omega(g_f^p+g_{u_k}^p)\,\dmu 
	\leq 2^p\int_\Omega(g_f^p+g_{u_j}^p)\,\dmu.
\end{align*}
By induction, and letting $j=n$, it follows that 
\[
	\int_\Omega g_{v_{n,n}}^p\,\dmu
	\leq \int_\Omega g_{u_n}^p\,\dmu 
	\quad\text{and}\quad
	\int_X g_{\phi_{n,n}}^p\,\dmu 
	\leq 2^p\int_\Omega (g_f^p+g_{u_n}^p)\,\dmu,\ \
	n=1,2,\dots.
\]

Fix an integer $m\geq 1$. 
Since $\{\phi_{n,n}\}_{n=1}^\infty$ 
and $\{g_{\phi_{n,n}}\}_{n=1}^\infty$ 
are bounded in $\Lp(mB)$ 
and $\phi_{n,n}\to \phi$ q.e.\ in $mB$ as $n\to\infty$, 
Corollary~\ref{1-cor-mazur-consequence} 
asserts that $\phi\in\Np(mB)$. 
This implies that $\phi\in\Dploc(X)$. 
Note that  
$g_\phi$ and $g_{\phi_{n,n}}$ are 
minimal \pp-weak upper gradients of 
$\phi$ and $\phi_{n,n}$, respectively, with respect to $mB$. 
Hence, by Corollary~\ref{1-cor-mazur-consequence} again, 
it follows that 
\begin{align*}
	\int_{mB}g_\phi^p\,\dmu 
	&\leq \liminf_{n\to\infty}\int_{mB}g_{\phi_{n,n}}^p\,\dmu 
	\leq \liminf_{n\to\infty}\int_X g_{\phi_{n,n}}^p\,\dmu \\
	&\leq 2^p\liminf_{n\to\infty}\int_\Omega (g_f^p+g_{u_n}^p)\,\dmu 
	= 2^p\int_\Omega g_f^p\,\dmu + 2^p I.
\end{align*}
Letting $m\to\infty$ yields
\[
	\int_X g_\phi^p\,d\mu
	= \lim_{m\to\infty}\int_{mB} g_\phi^p\,d\mu
	\leq 2^p\int_\Omega g_f^p\,\dmu + 2^p I 
	< \infty,
\]
and hence $\phi\in\Dp(X)$. 
We conclude that $u-f\in\Dp_0(\Omega)$ and $u\in\Dp(\Omega)$.

Let $A_n=\{x\in\Omega:v_{n,n}(x)<\psi(x)\}$ for $n=1,2,\dots$\,,  
and let $A=\bigcup_{n=1}^\infty A_n$. 
Then, since $v_{n,n}\to u$ q.e.\ in $\Omega$ as $n\to\infty$, 
it follows that $u\geq\psi$ q.e.\ in $\Omega\setm A$. 
Because $v_{n,n}\geq\psi$ q.e.\ in $\Omega$, we have $\Cp(A_n)=0$, 
and hence $\Cp(A)=0$ by the subadditivity of the capacity.
Thus $u\geq\psi$ q.e.\ in $\Omega$, and we conclude that $u\in\K_{\psi,f}$.

Proposition~\ref{1-prop-Dploc-Nploc} asserts that $f\in\Nploc(\Omega)$, 
and hence $f\in\Lp(\Omega')$ for all open $\Omega'\Subset\Omega$.
Let 
\[
	\Omega_t 
	= \Bigl\{x\in tB\cap\Omega:
			\inf_{y\in\bdy\Omega}d(x,y)>\delta/t\Bigr\},\ \
	1\leq t<\infty,
\]
where $\delta>0$ is chosen small enough so that $\Omega_1$ is nonempty.
Then we have $\Omega_1\Subset\Omega_2\Subset\cdots\Subset\Omega
	=\bigcup_{t=1}^\infty\Omega_t$. 
Moreover, $\{v_{n,n}\}_{n=1}^\infty$ is bounded in $\Lp(\Omega_t)$, since 
\[
	\|v_{n,n}\|_{\Lp(\Omega_t)} 
	\leq \|\phi_{n,n}\|_{\Lp(\Omega_t)} + \|f\|_{\Lp(\Omega_t)} 
	\leq C_{tB}' + \|f\|_{\Lp(\Omega_t)}<\infty. 
\]

Fix an integer $m\geq 1$. 
Since $\{v_{n,n}\}_{n=1}^\infty$ and $\{g_{v_{n,n}}\}_{n=1}^\infty$ 
are bounded in $\Lp(\Omega_m)$, 
$v_{n,n}\to u$ q.e.\ in $\Omega_m$ as $n\to\infty$, 
and $g_u$ and $g_{v_{n,n}}$ are minimal \pp-weak upper gradients of 
$u$ and $v_{n,n}$, respectively, with respect to $\Omega_m$, 
by Corollary~\ref{1-cor-mazur-consequence}, 
it follows that 
\[
	\int_{\Omega_m}g_u^p\,\dmu 
	\leq \liminf_{n\to\infty}\int_{\Omega_m}g_{v_{n,n}}^p\,\dmu 
	\leq \liminf_{n\to\infty}\int_\Omega g_{v_{n,n}}^p\,\dmu 
	\leq \liminf_{n\to\infty}\int_\Omega g_{u_n}^p\,\dmu
	= I.
\]
Letting $m\to\infty$ completes 
the existence part of the proof by showing that 
\[
	I\leq 
	\int_\Omega g_u^p\,d\mu
	= \lim_{m\to\infty}\int_{\Omega_m} g_u^p\,d\mu
	\leq I.
\]

\medskip

\noindent
(Uniqueness.) 
Suppose that $u'$ and $u''$ are solutions 
to the $\K_{\psi,f}$-obstacle problem. 
We begin this part by showing that $g_{u'}=g_{u''}$ a.e.\ in $\Omega$. 

Clearly, $\frac{1}{2}(u'+u'')\in\K_{\psi,f}$, and hence  
\begin{align*}
	\|g_{u'}\|_{\Lp(\Omega)} 
	&\leq \|g_{\frac{1}{2}(u'+u'')}\|_{\Lp(\Omega)}
	\leq \bigl\|\tfrac{1}{2}(g_{u'}+g_{u''})\bigr\|_{\Lp(\Omega)} \\
	&\leq \tfrac{1}{2}\|g_{u'}\|_{\Lp(\Omega)}
		+\tfrac{1}{2}\|g_{u''}\|_{\Lp(\Omega)}
	= \|g_{u''}\|_{\Lp(\Omega)}
	= \|g_{u'}\|_{\Lp(\Omega)}.
\end{align*}
Thus $g_{u'}=g_{u''}$ a.e.\ in $\Omega$ 
by the strict convexity of $\Lp(\Omega)$. 

Now we show that $g_{u'-u''}=0$ a.e.\ in $\Omega$.
Fix a real number $c$ and let 
\[
	u = \max\{u',\min\{u'',c\}\}.
\]
The following shows that $u\in\K_{\psi,f}$. 
Clearly, $u\in\Dp(\Omega)$. 
Furthermore, we have $u\geq u'\geq\psi$ q.e.\ in $\Omega$, 
and $u-f\in\Dp_0(\Omega)$ by Lemma~\ref{1-lem-Dp0-police}, 
since 
\[
	u-f \leq \max\{u',u''\}-f
	= \max\{u'-f,u''-f\} 
	\in\Dp_0(\Omega)
\]
and $u-f\geq u'-f\in\Dp_0(\Omega)$. 

Let $U_c=\{x\in\Omega:u'(x)<c<u''(x)\}$. 
Then we have $g_u=0$ a.e\ in $U_c$, 
since $U_c\subset\{x\in\Omega:u(x)=c\}$. 
The minimizing property of $g_{u'}$ then implies that 
\[
	\int_\Omega g_{u'}^p\,\dmu
	\leq \int_\Omega g_u^p\,\dmu
	= \int_{\Omega\setm U_c}g_u^p\,\dmu
	= \int_{\Omega\setm U_c}g_{u'}^p\,\dmu,
\]
since $g_u=g_{u'}=g_{u''}$ a.e.\ in $\Omega\setm U_c$. 
Hence $g_{u'}=g_{u''}=0$ a.e.\ in $U_c$ for all $c\in\R$, and because 
\[
	\{x\in\Omega:u'(x)<u''(x)\} \subset \bigcup_{c\in\Q}U_c,
\]
we have $g_{u'}=g_{u''}=0$ a.e.\ in $\{x\in\Omega:u'(x)<u''(x)\}$. 
Analogously, the same is true for $\{x\in\Omega:u'(x)>u''(x)\}$, and hence 
\[
	g_{u'-u''}
	\leq (g_{u'}+g_{u''})
		\chi_{\{x\in\Omega:u'(x)\neq u''(x)\}}
	= 0\ \
	\textup{a.e.\ in }\Omega.
\]

Since $u'-u''=u'-f-(u''-f)\in\Dp_0(\Omega)$, 
there exists $w\in\Dp(X)$ such that 
$w=u'-u''$ in $\Omega$ and $w=0$ outside $\Omega$. 
We have $g_w=g_{u'-u''}\chi_\Omega=0$ a.e.

Let $\widetilde{S}=\{x\in X:w(x)=0\}$ and let $t\geq 1$ be arbitrary. 
Then 
\[
	\Cp(tB\cap\widetilde{S})
	\geq \Cp(tB\setm\Omega)
	\geq \Cp(B\setm\Omega)
	> 0.
\]
Maz\cprime ya's inequality (Theorem~\ref{1-thm-Mazyas-ineq}) applies to $w$, 
and hence there exists a constant $\widetilde{C}_{tB}>0$ such that
\[
	\int_{tB\cap\Omega}|u'-u''|^p\,\dmu 
	\leq \int_{2tB}|w|^p\,\dmu 
	\leq \widetilde{C}_{tB}\int_{2\lambda tB}g_{w}^p\,\dmu 
	= 0.
\]
This implies that $u'=u''$ q.e.\ in $tB\cap\Omega$. 

Let $V_m=\{x\in mB\cap\Omega:u'(x)\neq u''(x)\}$, $m=1,2,\dots$\,, 
and let $V=\bigcup_{m=1}^\infty V_m$. 
Then $u'=u''$ in $\Omega\setm V$. 
Since $\Cp(V_m)=0$ for all $m$, 
the subadditivity of the capacity implies that $\Cp(V)=0$, 
hence $u'=u''$ q.e.\ in $\Omega$.
We conclude that the solution of the $\K_{\psi,f}$-obstacle problem is unique 
(up to sets of capacity zero).
\end{proof}
%===============================================================================
If $v=u$ q.e.\ in $\Omega$ and $u$ is a solution of the 
$\K_{\psi,f}$-obstacle problem, then so is $v$. 
Indeed, $v=u$ q.e.\ implies that  
$g_u$ is a \pp-weak upper gradient of $v$. 
Thus $v\in\Dp(\Omega)$ and 
$\int_\Omega g_v^p\,\dmu\leq\int_\Omega g_u^p\,\dmu$. 
Clearly, we have $v\geq\psi$ q.e., 
and since Lemma~\ref{1-lem-Dp0-police} asserts that 
$v-f\in\Dp_0(\Omega)$, 
it follows that $v\in\K_{\psi,f}$.

\medskip

The following criterion 
for the existence of a unique solution is easy to prove.
%===============================================================================
\begin{proposition} \label{1-prop-obst-solve-crit}
Suppose that $\psi$ and $f$ are in $\Dp(\Omega)$.
Then $\K_{\psi,f}$ is nonempty 
if and only if 
$(\psi-f)_\limplus\in\Dp_0(\Omega)$.
\end{proposition}
%-------------------------------------------------------------------------------
\begin{proof}
Suppose that $\K_{\psi,f}$ is nonempty and 
let $v\in\K_{\psi,f}$. 
Since $(v-f)_\limplus\in\Dp_0(\Omega)$ and  
\[
	0
	\leq (\psi-f)_\limplus
	\leq (v-f)_\limplus\ \ 
	\text{q.e.\ in }\Omega, 
\]
Lemma~\ref{1-lem-Dp0-police} asserts that 
$(\psi-f)_\limplus\in\Dp_0(\Omega)$.

Conversely, 
suppose that $(\psi-f)_\limplus\in\Dp_0(\Omega)$. 
Let $v=\max\{\psi,f\}$. 
Then we have $v\in\Dp(\Omega)$, 
$v-f=(\psi-f)_\limplus$, 
and $v\geq\psi$ in $\Omega$.
Thus $v\in\K_{\psi,f}$.
\end{proof}
%===============================================================================
The following comparison principle 
(for the version of the obstacle problem defined in 
Kinnunen--Martio~\cite{KiMa02}) 
was obtained in Bj\"orn--Bj\"orn~\cite{BjBj06}. 
Their proof (with trivial modifications) 
is valid also for our obstacle problem. 
%===============================================================================
\begin{lemma} \label{1-lem-obst-le}
Let $\psi_j\colon\Omega\to\eR$ and 
$f_j\in\Dp(\Omega)$ be such that $\K_{\psi_j,f_j}$ is nonempty\textup{,} 
and let $u_j$ be a solution of the $\K_{\psi_j,f_j}$-obstacle problem 
for $j=1,2$. 
If $\psi_1\leq\psi_2$ q.e\ in\/ $\Omega$ and 
$(f_1-f_2)_\limplus\in\Dp_0(\Omega)$\textup{,} 
then $u_1\leq u_2$ q.e.\ in\/ $\Omega$. 
\end{lemma}
%-------------------------------------------------------------------------------
\begin{proof}
Let $h=u_1-f_1-u_2+f_2$. 
Then $h\in\Dp_0(\Omega)$ and  
\begin{align*}
	-(f_1-f_2)_\limplus - h_\limminus 
	&= -\max\{-(f_2-f_1),0\} - \max\{-h,0\} \\
	&= \min\{f_2-f_1,0\} + \min\{h,0\} 
	\leq \min\{f_2-f_1,h\} 
	\leq h.
\end{align*}
Since $-(f_1-f_2)_\limplus - h_\limminus\in\Dp_0(\Omega)$, 
Lemma~\ref{1-lem-Dp0-police} asserts that  
$\min\{f_2-f_1,h\}\in\Dp_0(\Omega)$. 

Let $u=\min\{u_1,u_2\}$. 
Then $u\in\Dp(\Omega)$, 
and since $u_2\geq\psi_2\geq\psi_1$ q.e.\ in $\Omega$, 
we clearly have $u\geq\psi_1$ q.e.\ in $\Omega$. 
Moreover, since $u_1-f_1=u_2-f_2+h$, 
we have 
\begin{align*}
	u - f_1 
	&= \min\{u_1,u_2\}-f_1 
	= \min\{u_1-f_1,u_2-f_1\} \\
	&= \min\{u_2-f_2+h,u_2-f_1\} 
	= u_2-f_2 + \min\{h,f_2-f_1\}.
\end{align*}
Hence $u-f_1\in\Dp_0(\Omega)$, 
and we conclude that $u\in\K_{\psi_1,f_1}$.

Let $v=\max\{u_1,u_2\}$. 
Then $v\in\Dp(\Omega)$ and $v\geq\psi_2$ q.e.\ in $\Omega$.
As 
\begin{align*}
	v-f_2
	&= \max\{u_1-f_2,u_2-f_2\} 
	= \max\{u_1-f_2,u_1-f_1-h\} \\
	&= u_1-f_1+\max\{f_1-f_2,-h\}  
	= u_1-f_1-\min\{f_2-f_1,h\}, 
\end{align*}
we see that $v-f_2\in\Dp_0(\Omega)$, 
and hence $v\in\K_{\psi_2,f_2}$.

Let $E=\{x\in\Omega:u_2(x)\leq u_1(x)\}$. 
Since $u_2$ is a solution 
of the $\K_{\psi_2,f_2}$-obstacle problem, 
we have  
\[
	\int_\Omega g_{u_2}^p\,\dmu
	\leq \int_\Omega g_v^p\,\dmu 
	= \int_E g_{u_1}^p\,\dmu 
		+\int_{\Omega\setm E}g_{u_2}^p\,\dmu, 
\]
which implies that 
\[
	\int_E g_{u_2}^p\,\dmu 
	\leq \int_E g_{u_1}^p\,\dmu.
\]
By using the last inequality, we see that  
\[
	\int_\Omega g_u^p\,\dmu 
	= \int_E g_{u_2}^p\,\dmu
		+\int_{\Omega\setm E}g_{u_1}^p\,\dmu 
	\leq \int_E g_{u_1}^p\,\dmu 
		+\int_{\Omega\setm E}g_{u_1}^p\,\dmu
	= \int_\Omega g_{u_1}^p\,\dmu. 
\]
Since $u\in\K_{\psi_1,f_1}$ and $u_1$ is a solution 
of the $\K_{\psi_1,f_1}$-obstacle problem, 
this inequality implies that also $u$ is a solution. 
Theorem~\ref{1-thm-obst-existence-uniqueness} 
asserts that $u_1=u$ q.e.\ in $\Omega$, 
and we conclude that $u_1\leq u_2$ q.e.\ in $\Omega$.
\end{proof}
%===============================================================================
The following local property of solutions 
of the obstacle problem can be useful. 
In some cases it may enable the use of results 
from the theory for bounded sets. 
In this paper, 
we will use it in the proof of Theorems~\ref{1-thm-obst-solve-cont} 
and \ref{1-thm-obst-solve-cont-pw}.
%===============================================================================
\begin{proposition} \label{1-prop-obst-solve-subset}
Let $\psi\colon\Omega\to\eR$ and 
$f\in\Dp(\Omega)$ be such that $\K_{\psi,f}$ is nonempty\textup{,} 
and let $u$ be a solution of the $\K_{\psi,f}$-obstacle problem. 
Suppose that\/ $\Omega'$ is an open subset of\/ $\Omega$. 
Then $u$ is a solution of the $\K_{\psi,u}(\Omega')$-obstacle problem.

Moreover\textup{,} 
if\/ $\Omega'\Subset\Omega$\textup{,} 
then $u$ is a solution also of the $\Kb_{\psi,u}(\Omega')$-obstacle problem.
\end{proposition}
%-------------------------------------------------------------------------------
\begin{proof}
Let $\Omega'$ be an open subset of $\Omega$. 
Clearly, $u\in\K_{\psi,u}(\Omega')$. 
Let $v\in\K_{\psi,u}(\Omega')$ be arbitrary. 
To complete the first part of the proof, 
it is sufficient to show that 
\begin{equation} \label{1-prop-obst-solve-subset-ineq}
	\int_{\Omega'}g_u^p\,\dmu 
	\leq \int_{\Omega'}g_v^p\,\dmu.
\end{equation}

Let $E=\Omega\setm\Omega'$ and extend $v$ to $\Omega$ 
by letting $v=u$ in $E$. 
Since $v-u\in\Dp_0(\Omega')$, 
we have $v=(v-u)+u\in\Dp(\Omega)$ and 
$v-f=(v-u)+(u-f)\in\Dp_0(\Omega)$, 
and since $v\geq\psi$ q.e.\ in $\Omega'$ 
and $v=u\geq\psi$ q.e.\ in $E$, 
we conclude that $v\in\K_{\psi,f}$.

Because $u$ is a solution to the $\K_{\psi,f}$-obstacle problem, we have 
\begin{equation} \label{1-prop-obst-solve-subset-ineq2}
	\int_{\Omega'}g_u^p\,\dmu 
	+ \int_E g_u^p\,\dmu 
	= \int_{\Omega}g_u^p\,\dmu 
	\leq \int_{\Omega}g_v^p\,\dmu 
	= \int_{\Omega'}g_v^p\,\dmu 
	+ \int_E g_v^p\,\dmu.
\end{equation}
Since $u=v$ in $E$ implies that $g_u=g_v$ a.e.\ in $E$, we have 
\[
	\int_E g_v^p\,\dmu 
	= \int_E g_u^p\,\dmu 
	\leq \int_\Omega g_u^p\,\dmu 
	< \infty.
\] 
Subtracting the integrals over $E$ in 
\eqref{1-prop-obst-solve-subset-ineq2} 
yields 
\eqref{1-prop-obst-solve-subset-ineq}.

For the second part, 
assume that $\Omega'\Subset\Omega$ and 
let $v\in\Kb_{\psi,u}(\Omega')$ be arbitrary. 
Clearly, \smash{$v\in\K_{\psi,u}(\Omega')$}. 
The first part of the proof asserts that $u$ is a solution of the 
\smash{$\K_{\psi,u}(\Omega')$}-obstacle problem 
and hence \eqref{1-prop-obst-solve-subset-ineq} holds. 
By Proposition~\ref{1-prop-Dploc-Nploc}, 
we have \smash{$u\in\Nploc(\Omega)$}, 
and hence \smash{$u\in\Np(\Omega')$}. 
Thus \smash{$u\in\Kb_{\psi,u}(\Omega')$} 
and the proof is complete.
\end{proof}
%===============================================================================
There are many equivalent definitions of (super)minimizers in the literature 
(see Proposition~3.2 in A.~Bj\"orn~\cite{BjornA06a}). 
The first definition for metric spaces was given by 
Kinnunen--Martio~\cite{KiMa02}. 
Here we follow 
Bj\"orn--Bj\"orn--M\"ak\"al\"ainen--Parviainen~\cite{BjBjMaPa09}. 
We also follow the custom of not making the dependence on $p$ 
explicit in the notation. 
%===============================================================================
\begin{definition} \label{1-def-min}
Let $V$ be a nonempty open subset of $X$. 
We say that a function $u\in\Nploc(V)$ 
is a \emph{superminimizer} in $V$ if 
\begin{equation} \label{1-def-min-eq}
	\int_{\phi\neq 0}g_u^p\,\dmu 
	\leq \int_{\phi\neq 0}g_{u+\phi}^p\,\dmu
\end{equation}
holds for all nonnegative $\phi\in\Np_0(V)$.

Furthermore, $u$ is said to be a \emph{minimizer} in $V$ 
if \eqref{1-def-min-eq} holds for all $\phi\in\Np_0(V)$. 
\end{definition}
%===============================================================================
According to Proposition~3.2 in A.~Bj\"orn~\cite{BjornA06a}, 
it is in fact only necessary to test \eqref{1-def-min-eq} with 
(nonnegative and all, respectively) $\phi\in\Lipc(V)$.

As a direct consequence of Proposition~\ref{1-prop-obst-solve-subset} 
together with Proposition~9.25 in Bj\"orn--Bj\"orn~\cite{Boken}, 
we have the following result.
%===============================================================================
\begin{proposition} \label{1-prop-obst-super}
Suppose that $u$ is a solution of the $\K_{\psi,f}$-obstacle problem. 
Then $u$ is a superminimizer in\/ $\Omega$.
\end{proposition}
%===============================================================================

\medskip

%===============================================================================
\section{Lsc-regularized solutions and \pp-harmonic solutions} 
\label{1-section-lsc-harm}
%===============================================================================
\emph{In this section\textup{,} 
we make the rather standard assumptions that\/ 
$1<p<\infty$\textup{,} 
that $X$ is a complete \pp-Poincar\'e space\textup{,} 
that $\mu$ is doubling\textup{,} 
and that\/ $\Omega$ is a nonempty open subset of $X$ 
such that $\Cp(X\setm\Omega)>0$.}

\medskip

When $\mu$ is doubling, 
it is true that $X$ is proper if and only if $X$ is complete, 
and also that $X$ supports a $(p,p)$-Poincar\'e inequality 
if and only if $X$ supports a \pp-Poincar\'e inequality 
(the necessity follows from H\"older's inequality, 
and the sufficiency was proved in Haj\l{}asz--Koskela~\cite{HaKo95}; 
see also Corollary~4.24 in Bj\"orn--Bj\"orn~\cite{Boken}).
Thus, the difference between this section and the previous 
is that here we make the assumption that $\mu$ is doubling. 

Note that under these assumptions, 
Poincar\'e inequalities are self-improving 
in the sense that 
$X$ supports a $q$-Poincar\'e inequality for some $q<p$ 
(this was proved by Keith--Zhong~\cite{KeZh08}). 
Hence, in this section, 
we make the same assumptions as Kinnunen--Martio~\cite{KiMa02}, 
and we can therefore use Theorems~5.1 and 5.5 in \cite{KiMa02}.

\medskip

%===============================================================================
\begin{theorem} \label{1-thm-obst-solve-reg} 
Let $\psi\colon\Omega\to\eR$ and let $f\in\Dp(\Omega)$. 
Then there exists a unique lsc-regularized solution of the 
$\K_{\psi,f}$-obstacle problem 
whenever $\K_{\psi,f}$ is nonempty.
\end{theorem}
%-------------------------------------------------------------------------------
The \emph{lsc-regularization} of a function $u$ is 
the (lower semicontinuous) function $u^*$ defined by 
\[
	u^*(x) 
	:= \essliminf_{y\to x}u(y) 
	:= \lim_{r\to 0}\essinf_{B(x,r)}u.
\]
%-------------------------------------------------------------------------------
\begin{proof}
Suppose that $\K_{\psi,f}$ is nonempty. 
Theorem~\ref{1-thm-obst-existence-uniqueness} asserts that there exists 
a solution $u$ of the $\K_{\psi,f}$-obstacle problem 
and that all solutions are equal to $u$ q.e.\ in $\Omega$.  
Proposition~\ref{1-prop-obst-super} asserts that  
$u$ is a superminimizer in $\Omega$, 
and hence by Theorem~5.1 in Kinnunen--Martio~\cite{KiMa02}, 
we have $u^*=u$ q.e.\ in $\Omega$.  
Thus $u^*$ is the unique lsc-regularized 
solution of the $\K_{\psi,f}$-obstacle problem. 
\end{proof}
%===============================================================================
The following comparison principle improves upon Lemma~\ref{1-lem-obst-le}.
%===============================================================================
\begin{lemma} \label{1-lem-obst-reg-le}
Let $\psi_j\colon\Omega\to\eR$ and 
$f_j\in\Dp(\Omega)$ be such that $\K_{\psi_j,f_j}$ is nonempty\textup{,} 
and let $u_j$ be the lsc-regularized solution 
of the $\K_{\psi_j,f_j}$-obstacle problem 
for $j=1,2$. 
Then $u_1\leq u_2$ in\/ $\Omega$ 
whenever $\psi_1\leq\psi_2$ q.e\ in\/ $\Omega$ 
and $(f_1-f_2)_\limplus\in\Dp_0(\Omega)$. 
\end{lemma}
%-------------------------------------------------------------------------------
\begin{proof}
By Lemma~\ref{1-lem-obst-le}, 
we have $u_1\leq u_2$ q.e.\ in $\Omega$, 
and since both $u_1$ and $u_2$ are lsc-regularized, 
it follows that 
\[
	u_1(x) 
	= \essliminf_{y\to x}u_1(y) 
	\leq \essliminf_{y\to x}u_2(y) 
	= u_2(x)\ \ 
	\text{for all }x\in\Omega.
	\qedhere
\]
\end{proof}
%===============================================================================
\begin{definition} \label{1-def-p-harmonic}
Let $V$ be a nonempty open subset of $X$. 
We say that a function $u\in\Nploc(V)$ is \emph{\pp-harmonic} in $V$ 
whenever it is a continuous minimizer in $V$.
\end{definition}
%===============================================================================

\medskip

Kinnunen and Martio proved that the solution $u$ (if it exists) 
of their obstacle problem 
for bounded sets is continuous in $\Omega$ 
and is a minimizer in the open set $\{x\in\Omega:u(x)>\psi(x)\}$ 
whenever the obstacle $\psi$ is continuous in $\Omega$ 
(Theorem~5.5 in \cite{KiMa02}). 
This is true also for the $\Kb_{\psi,f}(\Omega)$-obstacle problem 
(see, e.g., Theorem~8.28 in Bj\"orn--Bj\"orn~\cite{Boken}), 
and also for our obstacle problem (that allows for unbounded sets).
%===============================================================================
\begin{theorem} \label{1-thm-obst-solve-cont} 
Let $\psi\colon\Omega\to[-\infty,\infty)$ be continuous and 
$f\in\Dp(\Omega)$ be such that $\K_{\psi,f}$ is nonempty. 
Then the lsc-regularized solution $u$ of the $\K_{\psi,f}$-obstacle problem 
is continuous in\/ $\Omega$ 
and \pp-harmonic in the open set 
$A=\{x\in\Omega:u(x)>\psi(x)\}$. 
\end{theorem}
%===============================================================================
We also have the following corresponding pointwise result.
%===============================================================================
\begin{theorem} \label{1-thm-obst-solve-cont-pw} 
Let $\psi\colon\Omega\to[-\infty,\infty)$ and 
$f\in\Dp(\Omega)$ be such that $\K_{\psi,f}$ is nonempty. 
Let $x\in\Omega$. 
Then the lsc-regularized solution $u$ of the $\K_{\psi,f}$-obstacle problem 
is continuous at $x$ 
if $\psi$ is continuous at $x$. 
\end{theorem}
%-------------------------------------------------------------------------------
\begin{proof}
Let $x\in\Omega$ and let $\Omega'$ be an open set 
such that $x\in\Omega'\Subset\Omega$. 
Let $u$ be the lsc-regularized solution of the $\K_{\psi,f}$-obstacle problem. 
Proposition~\ref{1-prop-obst-solve-subset} asserts that 
$u$ is a solution of the $\Kb_{\psi,u}(\Omega')$-obstacle problem. 
By Theorem~8.29 in Bj\"orn--Bj\"orn~\cite{Boken} 
(which is a special case of Corollary~3.4 in Farnana~\cite{Farnana11}), 
it follows that $u$ is continuous at $x$.
\end{proof}
%-------------------------------------------------------------------------------
\begin{proof}[Proof of Theorem~\ref{1-thm-obst-solve-cont}]
The first part follows directly from Theorem~\ref{1-thm-obst-solve-cont-pw}. 

Now we prove that $u$ is a minimizer in $A$. 
The set $A$ is open since $\psi$ and $u$ are continuous. 
Choose a ball $B\subset A$ and $\delta>0$ small enough so that the sets 
\[
	A_n 
	:= \Bigl\{x\in nB\cap A:
			\inf_{y\in\bdy A}d(x,y)>\delta/n\Bigr\},\ \
	n=1,2,\dots,
\]
are nonempty. 
Then $A_1\Subset A_2\Subset\cdots\Subset A=\bigcup_{n=1}^\infty A_n$. 
Fix a positive integer $n$. 
Since $u$ is a solution of the $\Kb_{\psi,u}(A_n)$-obstacle problem, 
Theorem~5.5 in Kinnunen--Martio~\cite{KiMa02} asserts that $u$ is 
\pp-harmonic in $A_n$. 
From this, 
it follows that $u$ is \pp-harmonic in $A$ 
(see, e.g., Theorem~9.36 in Bj\"orn--Bj\"orn~\cite{Boken}).
\end{proof}
%===============================================================================
Due to Theorem~\ref{1-thm-obst-solve-cont}, 
the following definition makes sense.
%===============================================================================
\begin{definition} \label{1-def-ext}
The \emph{\pp-harmonic extension} 
$\Hp_\Omega f$ of a function $f\in\Dp(\Omega)$ to $\Omega$ 
is the continuous solution 
of the $\K_{-\infty,f}(\Omega)$-obstacle problem.
\end{definition}
%===============================================================================
Then $\Hp_\Omega f$ is the unique \pp-harmonic function in $\Omega$ 
such that $f-\Hp_\Omega f\in\Dp_0(\Omega)$. 
Note that Definition~\ref{1-def-ext} is a generalization of 
Definition~8.31 in Bj\"orn--Bj\"orn~\cite{Boken} 
to Dirichlet functions and to unbounded sets 
(see Remark~\ref{1-rem-extension}).

We conclude that we have solved the Dirichlet problem 
for \pp-harmonic functions 
in open sets with boundary values in $\Dp(\Omega)$ taken in Sobolev sense, 
and we finish the paper by giving a short proof 
of the following comparision principle.
%===============================================================================
\begin{lemma} \label{1-lem-comp-principle}
Suppose that $f_1$ and $f_2$ are in $\Dp(\Omega)$ 
and that $(f_1-f_2)_\limplus\in\Dp_0(\Omega)$. 
Then $\Hp_\Omega f_1\leq\Hp_\Omega f_2$ in\/ $\Omega$.

The conclusion holds also under the assumption that 
$f_1$ and $f_2$ belong to $\Dp(\overline\Omega)$ 
and that $f_1\leq f_2$ q.e.\ on $\bdy\Omega$.
\end{lemma}
%-------------------------------------------------------------------------------
The first part is just a special case of 
Lemma~\ref{1-lem-obst-reg-le}.
%-------------------------------------------------------------------------------
\begin{proof}
We prove the second part. 
Clearly, 
$(f_1-f_2)_\limplus\in\Dp(\overline{\Omega})$.
Since $f_1\leq f_2$ q.e.\ on $\bdy\Omega$, 
we have $(f_1-f_2)_\limplus=0$ q.e.\ on $\overline{\Omega}\setm\Omega$, 
and hence $(f_1-f_2)_\limplus\in\Dp_0(\Omega;\overline{\Omega})$. 
Since $\Dp_0(\Omega)=\Dp_0(\Omega;\overline\Omega)$ 
according to Proposition~\ref{1-prop-Dp0-Om-clOm}, 
the result follows from the first part.
\end{proof}
%===============================================================================
% BIBLIOGRAPHY
%  \bibauthor{Name, A.}, 
%  \bibjtitle{Title}, 
%  \bibjournal{Journal Name} \bibvol{1} (year), pages.

%  \bibauthor{Name, B.}, 
%  \bibbtitle{Title}, 
%  Publisher, City, year.
%===============================================================================
\newcommand{\bibauthor}[1]{\textsc{#1}}
\newcommand{\bibjtitle}[1]{\textrm{#1}}
\newcommand{\bibbtitle}[1]{\textit{#1}}
\newcommand{\bibjournal}[1]{\textit{#1}}
\newcommand{\bibvol}[1]{\textrm{#1}}
\newlength{\bibindent}
\settowidth{\bibindent}{--- }

\end{document}